\documentclass[fleqn,a4paper,11pt]{article}
\usepackage{amsmath,amsfonts,calrsfs,amssymb,color,dsfont}
\usepackage{amsthm}

\newtheorem{Theorem}{Theorem}[section]

\newtheorem{Proposition}[Theorem]{Proposition}
\newtheorem{Lemma}[Theorem]{Lemma}

\newtheorem{Remark}[Theorem]{Remark}

\newtheorem{Hypothesis}[Theorem]{Hypothesis}
\makeatletter
\@addtoreset{equation}{section}

\makeatother

\def\R{\mathbb R}
\def\N{\mathbb N}

\def\E{\mathbb E}

\def\ds{\displaystyle}

\title{\bf An integral inequality  for the   invariant measure  of  some finite dimensional stochastic differential equation}
\author{Giuseppe Da Prato
\\\normalsize Scuola Normale Superiore di Pisa\\\\
Dedicated to Bj\"orn Schmalfuss}
\date{ }

\begin{document}
\maketitle

\begin{abstract}

We prove an integral inequality for the  invariant measure $\nu$  of a stochastic differential equation with additive noise in a finite dimensional space $H=\R^d$. As a consequence,   we show   that there exists  the  Fomin derivative of $\nu$ in any direction $z\in H$ and that it  is  given by $v_z=\langle D\log\rho,z\rangle$,  where $\rho$ is the density of $\nu$ with respect to the Lebesgue measure. Moreover, we prove that $v_z\in L^p(H,\nu)$ for any $p\in[1,\infty)$.  Also we study some properties of the gradient operator in $L^p(H,\nu)$ and of his adjoint.
\end{abstract}

\bigskip

\noindent {\bf 2000 Mathematics Subject Classification AMS}:  60H07,  60H30, 37L40.\medskip

 \bigskip 


\section{Introduction and preliminaries}

   In the  recent paper \cite{DaDe14}  the following  inequality   involving the invariant measure $\nu$  of the Burgers  equation    was proved
 \begin{equation}
\label{e1}
\left|\int_H \langle    RD\varphi,z\rangle\,d\nu   \right|\le C_p\|\varphi\|_{L^p(H,\nu)}\,|z|,
\end{equation}
 for all $\varphi\in C^1_b(H)$, all $z\in H$ and all $p>1$, $R$ being  a suitable negative power of the Laplace operator equipped with Dirichlet boundary conditions.

As noted in \cite{DaDe14}, by estimate \eqref{e1} it follows that  $RD$ is closable in $L^p(H,\nu)$ for all $p> 1$. Moreover, for each $z\in H$ there exists $v_z\in L^p(H,\nu)$ such that
\begin{equation}
\label{e2}
\int_H \langle    RD\varphi,z\rangle\,d\nu   =\int_H v_z\, \varphi\,d\nu,\quad \forall\;\varphi\in C^1_b(H).
\end{equation}
 Identity \eqref{e2}  implies   that $\nu$ is Fomin differentiable in all directions of the range of $R(H)$ of $R$.  We recall that if $\nu=N_Q$  (the Gaussian measure of mean  $0$ and covariance $Q$)  identity \eqref{e2} is well known in Malliavin Calculus. In this case the adjoint $(Q^{1/2}D)^*$ of $Q^{1/2}D$ is  called the Skorhood operator.\medskip

The aim of the present paper is to show that the inequality  \eqref{e1}, with $R$ replaced by the identity operator,  can also  be proved  for the invariant measures of  some stochastic differential equations in $H=\R^d$ of the form
 \begin{equation}
\label{e1.1}
\left\{\begin{array}{l}
dX(t)= b(X(t))dt+dW(t),\\\\
X(0)=x\in H,
\end{array}\right. 
\end{equation}
where $W$ is an $\R^d$--valued standard Brownian motion and $b$ fulfills the following assumptions.
 \begin{Hypothesis}
\label{h1}

(i) There exist  $\omega>0$,  $a\ge 0$ such that 
\begin{equation}
\label{e1.2}
\langle b(x),x   \rangle\le -\omega|x|^2+a,\quad \forall\;x\in \R^d,
\end{equation}

(ii) $b:H\to H$ is  continuously differentiable and  there exists  $K>0$, $N\in\N$ such that
\begin{equation}
\label{e1.3}
|b(x)|+\|b'(x)\|\le K(1+|x|^{2N}) , \quad\forall\;  
x\in \R^d.
\end{equation} 
 
 \end{Hypothesis}
\noindent  By   (ii)  it follows that $b$ is Lipschitz continuous on bounded sets of $H$, whereas (i) allows to estimate $|X(t,x)|^2$ by It\^o' formula; therefore  existence and uniqueness  of a strong solution $X(\cdot,x)$ of \eqref{e1.1} is classical, see e.g. the monograph \cite{Kr95}. 
 We shall denote by $P_t$ the transition semigroup
\begin{equation}
\label{e2.3a}
P_t\varphi(x)=\E[\varphi(X(t,x))],\quad t\ge 0,\; x\in H,\;\varphi\in B_b(H)
\end{equation}
For proving   \eqref{e1} we argue as in \cite{DaDe14} starting  from the elementary  identity, see \eqref{e40}
$$
  P_t(\langle D\varphi,h\rangle)= \langle DP_t\varphi, h\rangle  -   \int_0^tP_{t-s}(\langle Db\cdot h   ,D P_s\varphi \rangle)ds.
  $$
Then we prove  suitable  estimates for $DP_t\varphi$ and their integrals with respect to $\nu$. These estimates require some work  because, due to the polynomial growth  of the derivative of $b$, see \eqref{e1.3},  we cannot exploit the classical  Bismut--Elworthy--Li formula, see \cite{El92}. To overcome this problem we shall argue as in \cite{DaDe03}, \cite{DaDe07}  and  \cite{DaDe14}, introducing a suitable   potential (in the present case $ V(x)=K(1+|x|^{2N})$)
  and the Feynman--Kac semigroup
 \begin{equation}
\label{e1.10}
S_t\varphi(x)=\E[\varphi(X(t,x))\,e^{-\int_0^tV(X(s,x))\,ds}].
\end{equation}
We shall first  estimate $ \langle DS_t\varphi(x),h   \rangle$   then  $ \langle DP_t\varphi(x),h   \rangle$, by  taking  advantage of the identity
\begin{equation}
\label{e1.12}
P_t\varphi=S_t\varphi+\int_0^t S_{t-s}(VP_s\varphi)\,ds,
\end{equation}
which follows from the variation of constants formula,  see Section 2 below.\medskip

 In Section 3 we  prove  that inequality \eqref{e1}   and  identity \eqref{e2}   hold with $R=I$. 
  Moreover,  for any $z\in H$ we show that  the Fomin derivative $v_z$ in the direction $z\in H$ is given by $\langle D\log\rho,z   \rangle$, where $\rho$  is the  density  of $\nu$ with respect to  the Lebesgue measure. Moreover   $v_z\in L^p(H,\nu)$ for all  $p\in[1,\infty) $.   Finally, we prove a formula for  the adjoint $D^*$ of $D$ and also for the elliptic operator $-\tfrac12\,D^*D$ which can be seen as a generalisation of the Ornstein--Uhlenbeck operator.\medskip

   \bigskip
   
   We end this section with some notations.  We set $H=\R^d$, $d\ge 1$ (norm $|\cdot|$, inner product $\langle\cdot,\cdot \rangle$) and denote by  $L(H)$   the space of all linear bounded operators from $H$ into $H$. Moreover, $C_b(H)$   is  the space of all real continuous   and bounded mappings   $\varphi\colon H\to \R$ endowed with the sup norm 
  $$\|\varphi \|_{\infty}=\sup_{x\in  H}\,|\varphi(x)|$$
  whereas $C^k_b(H)$, $k>1,$ is  the space of all real functions which are continuous and bounded together with their derivatives of order lesser than $k$.
      Finally,   $ B_b(H)$ will represent the space of  all  real,  bounded   and Borel mappings on $H$.\medskip

\section{Estimates of the  derivative of the transition semigroup}

Let us start by giving an estimate of $\E(|X(t,x)|^{2m})$, $m\in\N$. The following lemma is standard, we shall give some details of the proof for the reader's convenience.
  \begin{Lemma}
\label{l2.1}
Assume Hypothesis \ref{h1}(i). Then
for any  $m\in\N$ there exists $a_m>0$ such that
\begin{equation}
\label{e2.1}
\E[|X(t,x)|^{2m}]\le e^{-2m\omega t}|x|^{2m}+a_m,\quad\forall\;x\in H,\;t\ge 0.
\end{equation}
\end{Lemma}
\begin{proof}
Let first  consider the case $m=1$. Then by It\^o's formula, taking into account \eqref{e1.2} we find
$$
\begin{array}{lll}
\ds\frac{d}{dt}\;\E[|X(t,x)|^{2}]&=&2\E[ \langle X(t,x), b(X(t,x))   \rangle]+d\\
\\
&\le& -2\omega\E[|X(t,x)|^{2}]+  2a+d.
\end{array}
$$
We deduce that
$$
\frac{d}{dt}\;\E\left[|X(t,x)|^{2}\right]\le -2\omega\E\left[|X(t,x)|^{2}\right]+2a +d.
$$
By a standard comparison result it follows that
\begin{equation}
\label{e2.2}
\E\left[|X(t,x)|^{2}\right]\le e^{-2\omega t}|x|^{2}+a_2,\quad\forall\;x\in H,\;t\ge 0,
\end{equation}
where
$$
a_2=\frac{1}{\omega}\,(2a+d).
$$
Now let $m>1$ and $\varphi_m(x)=|x|^{2m}$. Then we have
$$
D\varphi_m(x)=2m|x|^{2m-2}\,x
$$
and
$$
D^2\varphi_m(x)=4m(m-1)|x|^{2m-4}\,x\otimes x+2m|x|^{2m-2}\,I,
$$
where $I$ represents identity in $H$.
Consequently
$$
\frac12\;\mbox{\rm Tr}\;[D^2\varphi_m(x)]=m(2m-2+d)|x|^{2m-2}
$$
Then again by It\^o's formula we have
$$
\begin{array}{l}
\ds\frac{d}{dt}\;\E\left[|X(t,x)|^{2m}\right]=2m\E\left[X(t,x)|^{2m-2}\right] \langle X(t,x), b(X(t,x))   \rangle]\\
\\
\hspace{30mm}+m(2m-2+d)\E[X(t,x)|^{2m-2}]\\
\\
\ds \le -2m\omega\E\left[|X(t,x)|^{2m}\right] +m(2a+2m-2+d)\E\left[X(t,x)|^{2m-2}\right].
\end{array} 
$$
It
 follows that   
 $$
 \begin{array}{lll}
 \E[|X(t,x)|^{2m}]&\le&  e^{-2m\omega t}|x|^{2m}\\
 \\
 &&\ds+ m(2a+2m-2+d)\int_0^te^{-2m\omega (t-s)}\E[X(s,x)|^{2m-2}]ds.
 \end{array}
 $$
The conclusion follows easily by recurrence.
\end{proof}

Now we are going to  prove an estimate for the derivative $D_xX(t,x)h,$ which we denote by $\eta^h(t,x),\,h\in H.$ As  well known $\eta^h(t,x)$ is a solution to the random equation
\begin{equation}
\label{e1.8}
\left\{\begin{array}{l}
\ds \frac{d}{dt}\,\eta^h(t,x)=b'(X(t,x)\cdot\eta^h(t,x),\\
\\
\eta^h(0,x)=h
\end{array}\right. 
\end{equation}
 \begin{Lemma}
\label{l1.2}
Assume Hypothesis \ref{h1}. Then the  following estimate holds
\begin{equation}
\label{e1.9}
|\eta^h(t,x)|\le e^{K\int_0^t(1+|X(s,x)|^{2N})ds}\,|h|,\quad t\ge 0,\;x,h\in H.
\end{equation}
\end{Lemma}
\begin{proof}
By \eqref{e1.8} we deduce, taking into account
\eqref{e1.3}, that
$$
\frac12\,\frac{d}{dt}\,|\eta^h(t,x)|^2= \langle   b'(X(t,x)\cdot\eta^h(t,x),\eta^h(t,x)\rangle\le K(1+|X(t,x)|^{2N})\,|\eta^h(t,x)|^2.
$$
So, the conclusion follows from Gronwall's lemma.

\end{proof}

Now we  are going to estimate of $D_xP_t\varphi$. 

\subsection{Pointwise estimate}

As we said in the introduction,   we cannot  estimate  $D_xP_t\varphi$ for $\varphi\in C_b(H)$ using the   Bismut--Elworthy--Li  formula see \cite{El92},  because we do not know  whether the expectation on the right hand side of \eqref{e1.9} does exist.
  For this reason, we introduce the    potential
  $$
  V(x)=K(1+|x|^{2N}),\quad x\in H
  $$
  and the Feynman--Kac semigroup
$$
S_t\varphi(x)=\E[\varphi(X(t,x))\,e^{-\int_0^tV(X(s,x))\,ds}].
$$
We recall that the Bismut--Elworthy--Li formula generalises to $S_t$, see \cite{DaZa97}. In fact for all $\varphi\in C_b(H)$, setting
$$
 \beta(t)=\int_0^t V(X(s,x))ds,
 $$ the following identity holds
\begin{equation}
\label{e1.11}
\begin{array}{l}
\ds  \langle DS_t\varphi(x),h   \rangle=\frac1t\,\E \left[ \varphi(X(t,x))\,e^{-\beta(t)}\,\int_0^t  \langle\eta^h(s,x),dW(s)   \rangle \right]\\
\\
\ds-\E \left[ \varphi(X(t,x))\,e^{-\beta(t)}\int_0^t\left(1-\frac{s}{t}   \right)\,\langle    V'(X(s,x),\eta^h(s,x)  \rangle \right]\,ds\\
\\
=: I_1(\varphi,x,h,t)+I_2(\varphi,x,h,t)=I_1+I_2.
\end{array} 
\end{equation}
We shall first  estimate $ \langle DS_t\varphi(x),h   \rangle$,   then $ \langle DP_t\varphi(x),h   \rangle$. In the latter case, we take  advantage of the identity
$$
P_t\varphi=S_t\varphi+\int_0^t S_{t-s}(VP_s\varphi)\,ds,
$$
which follows from the variation of constants formula;  in  fact, denoting by $\mathcal L$ and
$\mathcal K$ the infinitesimal generators of $P_t$ and $S_t$ respectively, it holds
$$
\mathcal L=\mathcal K+V.
$$

  \begin{Lemma}
\label{l1.3}
Let  $\varphi\in  C_b(H)$, $t\ge 0$,  $x\in H$. Then for   $p>1$, there exists  a constant $C_p>0$ such that
\begin{equation}
\label{e13}
 |D_xS_t\varphi(x)| \le C_p(1+t^{-1/2})(1+|x|^{2N-1}) \left[\E\left( \varphi^p(X(t,x)\right)\right]^{1/p}.
\end{equation}
\end{Lemma}
\begin{proof} 
We start by estimating $ I_1.$
By  H\"older's inequality with  exponents $p,q=\tfrac{p}{p-1}$    we have
\begin{equation}
\label{e14e}
\begin{array}{l}
|I_1|\le  \ds\frac1t\left[\E\left( \varphi^p(X(t,x)\right)\right]^{1/p} \left[\E\left(e^{-q \beta(t)}\left|\int_0^t \langle   \eta^h(s,x),dW(s)\rangle\right|^q    \right)   \right]^{1/q}\\
\\
\hspace{6mm}=\ds:\frac1t\left[\E\left( \varphi^p(X(t,x)\right)\right]^{1/p} \left[\E\left(|z(t)|^q  \right)   \right]^{1/q},
\end{array} 
\end{equation}
where
\begin{equation}
\label{e15e}
z(t)=e^{- \beta(t)}\int_0^t \langle  \eta^h(s,x),dW(s)\rangle,\quad t\ge 0.
\end{equation}
We now apply It\^o's formula to $g(z(t))$ where
$g(r)=|r|^q,\;r\in\R.$ Since
$$
 g'(r)=q|r|^{q-2}r,\quad g''(r)=q(q-1)|r|^{q-2},
$$
and
$$
\begin{array}{lll}
dz(t)&=&\ds-\beta'(t)e^{- \beta(t)}\int_0^t \langle  \eta^h(s,x),dW(s)\rangle\,ds +e^{- \beta(t)}  \langle   \eta^h(t,x),dW(t)\rangle\\
\\
&=&\ds-\beta'(t)z(t)+e^{- \beta(t)}  \langle  \eta^h(t,x),dW(t)\rangle,
\end{array}
$$ 
we find
$$
\begin{array}{lll}
\ds d|z(t)|^q&=&q|z(t)|^{q-2}z(t)(-\beta'(t)z(t)+e^{- \beta(t)}  \langle   \eta^h(t,x),dW(t)\rangle)\\
\\
&&\ds+\frac12\,q(q-1)|z(t)|^{q-2}e^{-2\beta(t)}|  \eta^h(t,x)|^2dt.
\end{array} 
$$
Integrating from $0$ to $t$, yields
\begin{equation}
\label{e16f}
\begin{array}{lll}
 |z(t)|^q&=&\ds-q\int_0^t|z(s)|^{q} \beta'(s)\,ds
\\
\\
&&\ds+q\int_0^t|z(s)|^{q-2}z(s)  e^{- \beta(s)}  \langle   \eta^h(s,x),dW(s)\rangle\\
\\
&&\ds+\frac12\,q(q-1)\int_0^te^{-2\beta(s)}|z(s)|^{q-2}|  \eta^h(s,x)|^2ds.
\end{array} 
\end{equation}
Neglecting the negative first term in the previous identity  and taking expectation, we find 
\begin{equation}
\label{e16a}
\begin{array}{l}
\ds\E\left(\sup_{r\in [0,t]} |z(r)|^q    \right)\\
\\
\ds \le  q\,\E\left(\sup_{r\in[0,t]} \left|\int_0^r e^{- \beta(s)}|z(s)|^{q-2} z(s) \langle  \eta^h(s,x),dW(s)\rangle\right|\right)\\\\
\ds + \frac12q(q-1)\E\left(\int_0^t e^{-2 \beta(s)}|z(s)|^{q-2} |  \eta^h(s,x)|^2ds
\right)\\\\
=:A_1+A_2.
\end{array}
\end{equation}
By  the Burkholder inequality     we have, taking into account  Lemma \ref{l2.1}
\begin{equation}
\label{e16}
\begin{array}{l}
\ds A_1 \le \ds 3q \E\left[\left|\int_0^t e^{-2\beta(s)}|z(s)|^{2(q-1)} |  \eta^h(s,x)|^2ds\right|^{1/2}\right]\\
\ds \le 3q \E\left[\sup_{r\in [0,t]} |z(r)|^{q-1}\left(\int_0^t e^{-2 \beta(s)} | \eta^h(s,x)|^2ds\right)^{1/2}\right]\\
\\
\ds\le 3q t^{1/2}\,\E\left[\sup_{r\in [0,t]} |z(r)|^{q-1}\right]\,|h|.
\end{array} 
\end{equation}
By H\"older's inequality with exponents $q,\,\tfrac{q}{q-1},$ it follows that
\begin{equation}
\label{e23}
\begin{array}{l}
\ds A_1 \le   3q\,t^{1/2} |h|  \left[\E\left(\sup_{r\in [0,t]} |z(r)|^{q} \right)\right]^{\frac{q-1}{q}}.
\end{array} 
\end{equation}
Now by  the Young inequality 
\begin{equation}
\label{e23g}
ab\le \frac1{u}\,a^u+   \frac1{v}\,a^v,\quad a>0,\;b>0,\; \frac1{u}+ \frac1{v}=1
\end{equation}
with $u=q,\;v=\frac{q-1}{q}$,
there exists $c_1>0$ such that
\begin{equation}
\label{e24}
A_1\le \frac14\;\E\left(\sup_{r\in [0,t]} |z(r)|^{q} \right)+c_1\,t^{q/2} \,|h|^{q}.
\end{equation}

Concerning  $A_2$, using again Lemma \ref{l2.1}, we find
$$
\begin{array}{l}
\ds A_2= \frac12q(q-1)\E\left(\int_0^t e^{-2 \beta(s)}|z(s)|^{q-2} |\eta^h(s,x)|^2ds\right)\\
\\
\ds\le \frac12q(q-1)\;\E\left[\left(\sup_{r\in[0,t]} |z(r)|^{q-2}\right)\int_0^t e^{-2 \beta(s)}  |\eta^h(s,x)|^2ds  \right]\\
\\
\ds \le  \frac12q(q-1)\;\E\left[\left(\sup_{r\in[0,t]} |z(r)|^{q-2}\right)  \right]\,|h|^2\,t.
\end{array}
$$
By H\"older's inequality with exponents  
$\tfrac{q}2,\,\tfrac{q}{q-2}$ we have
$$
A_2\le \frac12\;q(q-1)\; |h|^2 \,\left[\E\left(\sup_{r\in[0,t]} |z(r)|^{q}\right)\right]^{\frac{q-2}{q}}.
$$
By the Young inequality \eqref{e23g} with $u=\frac{q}2$ and $v=\frac{q}{q-2}$, it follows that there exists $c_2>0$
 such that
\begin{equation}
\label{e26}
A_2\le \frac14\;\E\left(\sup_{r\in [0,t]} |z(r)|^{q} \right)+c_2 |h|^{q} \,t^{q/2}.
\end{equation}

Taking into account \eqref{e16a}, \eqref{e24} and \eqref{e26} we conclude that 
$$
\E\left(\sup_{r\in [0,t]} |z(r)|^{q} \right)\le  \frac12\;\E\left(\sup_{r\in [0,t]} |z(r)|^{q} \right)+(c_1+c_2) \,|h|^{q}\;t^{q/2}.
$$
Therefore
\begin{equation}
\label{e27}
\E\left(\sup_{r\in [0,t]} |z(r)|^{q} \right)\le  (c_1+c_2). \,|h|^{q}\;t^{q/2}.
\end{equation}
Finally, by \eqref{e14e} it follows that
\begin{equation}
\label{e28}
I_1\le (c_1+c_2) t^{-1/2} \,|h| \,\left[\E\left( \varphi^p(X(t,x)\right)\right]^{1/p}.
\end{equation}\medskip

Now let us consider $I_2$,
and write
\begin{equation}
\label{e29}
I_2
\le 2KN \left[\E\left[ \varphi^p(X(t,x)\right)\right]^{1/p}(\Lambda(t) )^{1/q},
\end{equation}
where $\frac1p+\frac1q=1$ and
\begin{equation}
\label{e30}
\begin{array}{l}
\ds\Lambda(t)=\E\left[e^{-q\beta(t)}\left(\int_0^t|X(s,x)|^{2N-1}  \,|\eta^h(s,x)|\,ds\right)^q \right]\\
\\
\ds\le\E\left[\left(\int_0^te^{-\beta(s)}|X(s,x)|^{2N-1}  \,|\eta^h(s,x)|\,ds\right)^q \right]  \\
\\
\ds\le \E\left[\sup_{r\in[0,t]}\left(|X(r,x)|^{(2N-1)q}\right) \left(\int_0^te^{-\beta(s)}\,|\eta^h(s,x)|\,ds\right)^q \right]\\
\\
\ds\le \E\left[\sup_{r\in[0,t]}\left(|X(r,x)|^{(2N-1)q}\right) \right]\,|h|^q\,t^q .
\end{array}
\end{equation} 
So
\begin{equation}
\label{e31}
I_2
\le 2KN\left[\E\left[ \varphi^p(X(t,x)\right)\right]^{1/p}  \left(\E\left[\sup_{r\in[0,t]}\left(|X(r,x)|^{(2N-1)q}\right) \right]\right)^{1/q} |h|\,t.
\end{equation}
Recalling finally \eqref{e2.1}  we  see that there exists $c_3>0$ such that
$$
\left(\E\left[\sup_{r\in[0,t]}\left(|X(r,x)|^{2N-1}\right) \right]\right)^q\le c_3(1+ |x|^{2N-1}),
$$
so that
\begin{equation}
\label{e32}
I_2
\le 2KNc_3(1+ |x|^{2N-1}) \left[\E\left[ \varphi^p(X(t,x)\right)\right]^{1/p}  \,|h| \,t.
\end{equation}
Finally, by \eqref{e1.11}, \eqref{e28} and \eqref{e32}, the conclusion follows easily.
\end{proof}
 
 \subsection{The invariant measure $\nu$}
  
We shall denote  by $\pi_{t,x}$ the law of $X(t,x)$ so that for each $\varphi\in B_b(H)$ we have
 \begin{equation}
\label{e2.4a}
P_t\varphi(x)=\int_H\varphi(y)\pi_{t,x}(dy),\quad x\in H,\;t>0.
\end{equation}

\begin{Lemma}
\label{l2.2}
Assume Hypothesis \ref{h1}(i). Then there is an  invariant measure $\nu$ of $P_t$, moreover for all $m\in\N$ we have
\begin{equation}
\label{e2.3}
\int_H|x|^{2m}\,\nu(dx)\le a_m,
\end{equation}
where $a_m$ is the constant in \eqref{e2.1}.
 \end{Lemma}
\begin{proof}
Let $r>0$ and fix $x\in H$. Set $B_r^c=\{y\in H:\; |y|\ge r\}$. Then, taking into account   \eqref{e2.2} it follows that
\begin{equation}
\label{enuova}
\begin{array}{lll}
 \pi_{t,x}(B_r^c)&=&\ds\int_{\{|y|\ge r\}}\pi_{t,x}(dy)\le \frac1{r^2}\int_H|y|^2\pi_{t,x}(dy)\\
\\
&=&\ds\frac1{r^2}\,\E\left[|X(t,x)|^2\right] \le \frac{|x|^2+a_2}{  r^2}.
\end{array}
\end{equation}
Therefore by the Krylov--Bogoliubov theorem, see e.g \cite{DaZa96},  there exists a sequence $T_n\uparrow +\infty$ such that
\begin{equation}
\label{e2.6a}
\lim_{n\to+\infty}\frac{1}{T_n}\int_0^{T_n}\pi_{t,x}dt=\nu\quad\mbox{\rm weakly},
\end{equation}
where $\nu$  is   an invariant measure of $P_t$.  

Now we can prove \eqref{e2.3}. By \eqref{e2.1} we deduce
\begin{equation}
\label{e2.7a}
\int_H |y|^{2m}\pi_{t,x}(dy)\le e^{-m\omega t}|x|^{2m}+a_m,\quad\forall\;x\in H,\;t\ge 0.
\end{equation}
It follows that for any $\epsilon>0$
\begin{equation}
\label{e2.8a}
\int_H\frac{ |y|^{2m}}{1+\epsilon|y|^{2m}}\,\pi_{t,x}(dy)\le e^{-m\omega t}|x|^{2m}+a_m,\quad\forall\;x\in H,\;t\ge 0.
\end{equation}
Consequently integrating both sides with respect to $t$ over $[0,T_n]$ and dividing by $T_n$, yields
\begin{equation}
\label{e2.8b}
\frac1{T_n}\int_0^{T^n}dt\int_H\frac{ |y|^{2m}}{1+\epsilon|y|^{2m}}\,\pi_{t,x}(dy)\le \frac1{m\omega T_n}\;(1-e^{-m\omega T_n})\,|x|^{2m}+a_m,
\end{equation}
 for all\ $x\in H,\;t\ge 0.$ Finally, letting $n\to+\infty$ and taking into account \eqref{e2.6a}, we find
$$
\int_H\frac{ |y|^{2m}}{1+\epsilon|y|^{2m}}\,\nu(dy)\le a_m
$$
and  the conclusion follows letting $\epsilon$ tend to $0$.

\end{proof}

\subsection{Integral estimates} 

Let us start with an estimate of 
$\int_H\langle D_xS_t\varphi(x),h(x)\rangle\,\nu(dx)$.
\begin{Lemma}
\label{l4}
Let   $p>1$, $q>1$,  $\frac1p+\frac1q<1$. Then there is $C_1>0$ such that
\begin{equation}
\label{e33}
\left| \int_H\langle D_xS_t\varphi(x),h(x)\rangle\,\nu(dx)\rangle  \right|\le  C_1(1+t^{-1/2}) \,\|\varphi\|_{L^p(H,\nu)}\;\|h\|_{L^q(H,\nu)}.
\end{equation}

\end{Lemma}
\begin{proof}
Taking into account \eqref{e13} we have
\begin{equation}
\label{e34}
\begin{array}{l}
\ds \left| \int_H\langle D_xS_t\varphi(x),h(x)\rangle\,\nu(dx) \right|\le C_p\,  (1+t^{-1/2})\\
\\
\ds\times\int_H(1+|x|^{2N-1}) \,\left[\left( |P_t\varphi^p(x)|\right)\right]^{1/p}\;|h(x)|\,\nu(dx).
\end{array}
\end{equation}
Let
$$
\frac1r=1-\frac1p-\frac1q,
$$
then by the triple  H\"older inequality  with exponents $r,p,q$ we have, taking into account the invariance of $\nu$,
\begin{equation}
\label{e35}
\begin{array}{l}
\ds \left| \int_H\langle D_xS_t\varphi(x),h(x)\rangle\,\nu(dx) \right|\le c\,  (1+t^{-1/2})\\
\\
\ds\times\left[\int_H(1+|x| ^{N-1})^r\nu(dx)\right]^{1/r}\,\left(\int_H|P_t \varphi^p(x)|\nu(dx)\right) ^{1/p}\;\|h \|_{L^q(H,\nu)}\\
\\
\ds\le c\,  (1+t^{-1/2})\left[\int_H(1+|x| ^{N-1})^r\nu(dx)\right]^{1/r}\,\;\|\varphi \|_{L^p(H,\nu)}\,\;\|h \|_{L^q(H,\nu)}.
\end{array}
\end{equation}
The conclusion follows from \eqref{e2.3}.
\end{proof}

 Now we are ready to estimate  $\int_H\langle D_xP_t\varphi(x),h(x)\rangle\,\nu(dx)$.
 We start from the identity
 \begin{equation}
\label{e36}
P_t\varphi(x)=S_t\varphi(x)+K\int_0^tS_{t-s}(1+|x|^{2N})\,P_s\varphi)(x)ds,
\end{equation}
from which
\begin{equation}
\label{e36a}
\begin{array}{l}
\langle D_xP_t\varphi(x),h(x)\rangle\\
\\
\ds=\langle D_xS_t\varphi(x),h(x)\rangle +K\int_0^t\langle D_x S_{t-s}((1+|x|^{2N})\,P_s\varphi)(x),h(x)\rangle ds.
\end{array}
\end{equation}
\begin{Proposition}
\label{p5}
Let   $p>1$, $q>1$,  $\frac1p+\frac1q<1$. Then there is $C^1_p$ such that
\begin{equation}
\label{e37g}
\left| \int_H\langle D_xP_t\varphi(x),h(x)\rangle\,\nu(dx)\rangle  \right|\le C^1_p(1+t^{-1/2})\,\|\varphi\|_{L^p(H,\nu)}\;\|h \|_{L^q(H,\nu)}.
\end{equation}

\end{Proposition}
\begin{proof}
The first term of \eqref{e36a} is bounded by \eqref{e33}.  Let us estimate the second one. 
Again by \eqref{e33} we have
\begin{equation}
\label{e39a}
\begin{array}{l}
\ds \left|\int_0^t\int_H\langle D_x S_{t-s}((1+|x|^{2N})\,P_s\varphi),h(x)\rangle ds\right|\,\nu(dx)\\
\\
\ds \le C_1\int_0^t (1+(t-s)^{-\frac12})\,\|(1+|x|^{2N})P_s\varphi\|_{L^p(H,\nu)}\ \;\|h\|_{L^q(H,\nu)}\,ds
\end{array} 
\end{equation}
Now let us  chose $\epsilon>0$ such that
$$
\frac1{p+\epsilon}+\frac1q<1.
$$
Then by H\"older's inequality with exponents
$\frac{p+\epsilon}\epsilon$ and $\frac{p+\epsilon}p$ it follows that
$$
\begin{array}{l}
\ds\|(1+|x|^{2N})\,P_s\varphi)\|^p_{L^p(H,\nu)}=\int_H
(1+|x|^{2N})^p\,(P_s\varphi)^p\,d\nu\\
\\
\ds\le \left(\int_H
(1+|x|^{2N})^{\frac{p(p+\epsilon)}\epsilon}\, d\nu   \right)^{\frac{\epsilon}{p+\epsilon}}\;\left(\int_H
 (P_s\varphi)^{p+\epsilon}\,d\nu   \right)^{\frac{p}{p+\epsilon}}\\\\
 \ds \le \left(\int_H
(1+|x|^{2N})^{\frac{p(p+\epsilon)}\epsilon}\, d\nu   \right)^{\frac{\epsilon}{p+\epsilon}}\;\|\varphi\|^p_{L^{p+\epsilon}(H,\nu)},
 \end{array} 
$$
by the invariance of $\nu$. 
Now by \eqref{e2.2} there exists a constant $C'$ such that
$$
 \int_H
(1+|x|^{2N})^{\frac{p(p+\epsilon)}\epsilon}\, d\nu \le C',
$$
Therefore
\begin{equation}
\label{e39g}
\|(1+|x|^{2N})\,P_s\varphi)\|^p_{L^p(H,\nu)}\le (C')^{\frac{\epsilon}{p+\epsilon}}\,\|\varphi\|^p_{L^{p+\epsilon}(H,\nu)}.
\end{equation}
Substituting in \eqref{e39a}, yields
\begin{equation}
\label{e39aa}
\begin{array}{l}
\ds \left|\int_0^t\langle D_x S_{t-s}(|x|^N_N\,P_s\varphi),h(x)\rangle ds\right|\\
\\
\ds \le C_1(C')^{\frac{\epsilon}{p+\epsilon}}\int_0^t (1+(t-s)^{-\frac{1 }2})\,\| \varphi\|_{L^{p+\epsilon}(H,\nu)}\;\|h|_{L^q(H,\nu)}\,ds
\end{array} 
\end{equation}

Non the conclusion follows   by the arbitrariness of $\epsilon, p, q$.
\end{proof}
  
 \section{The main inequality and its consequences}

\begin{Theorem}
\label{t6}
For all  $p>1$ there exists a constant $C_p>0$ such that for all $\varphi\in L^p(H,\nu)$ and all $h\in H$ we have
\begin{equation}
\label{e39}
\left|\int_H\langle D_x\varphi(x),h)\rangle\,\nu(dx)\right|\le c\|\varphi\|_{L^p(H,\nu)}\,|h| .
\end{equation}

\end{Theorem}
\begin{proof}
{\it Step 1}. For any $\varphi\in C^1_b(H)$ and any $h\in H$ the following   identity holds.
\begin{equation}
\label{e40}
 P_t(\langle D\varphi,h\rangle)= \langle DP_t\varphi, h\rangle  -   \int_0^tP_{t-s}(\langle Db\cdot H h  ,D P_s\varphi \rangle)ds,\quad t>0.
\end{equation}

To prove  \eqref{e40} we consider a sequence $(b_n)$ of mappings $H\to H$ of class $C^\infty$ such that\medskip

(i) $ \lim_{n\to\infty}b_n(x)=b(x),$  uniformly on bounded sets of $H$. \medskip

(ii) $\langle b_n(x),x\rangle\le -\omega|x|^2+a,\quad \forall\;x\in H. $\medskip

To construct $(b_n)$ we first set
$$
f_n(x)=\frac{b(x)+\omega x}{1+n^{-1}|x|^{2N+2}}-\omega x,
$$
so that
$$
\langle f_n(x),x\rangle\le -\omega|x|^2+a,\quad \forall\;x\in H,
$$
and $f_n$ is sub--linear, then we regularise $f_n$ using mollifiers. 

Now we prove the identity
\begin{equation}
\label{e40g}
 P^n_t(\langle D\varphi,h\rangle)= \langle DP^n_t\varphi, h\rangle  -   \int_0^tP_{t-s}(\langle Db\cdot h  ,D P^n_s\varphi \rangle)ds,
\end{equation}
where $P_t^n$ is the transition semigroup corresponding to $b_n$.
 
 It is enough to show \eqref{e40g} for each   $\varphi\in C^3_b(H)$. In such a case  set $u_{n}(t,x)=P^{n}_t\varphi(x)$ and write
\begin{equation}
\label{e3.10z}
\left\{\begin{array}{l}
\ds D_t u_{n}(t,x)=\frac12\,\Delta u_{n}(t,x)+ \langle Du_{n}(t,x),b_{n}(x)  \rangle,\\
\\
u_{n}(0,x)=\varphi(x).
\end{array}\right. 
\end{equation}
Now, taking $h\in H$ and setting
$$
v_{n}(t,x)=\langle D u_{n}(t,x),h   \rangle
$$
we see, by a simple computation, that
\begin{equation}
\label{e3.11z}
\left\{\begin{array}{lll}
D_t v_{n}(t,x)&=&\ds \frac12\,\Delta v_{n}(t,x)+ \langle Dv_{n}(t,x), b_{n}(x)  \rangle\\
\\
&&\ds+ \langle Du_{n}(t,x),b'_{n}(x)h  \rangle,\\
\\
v_{b}(0,x)&=& \langle D\varphi(x),h   \rangle.
\end{array}\right. 
\end{equation}
By the variation of constants formula it follows that
\begin{equation}
\label{e3.12z}
v_n(t,x)=P_t^n(\langle D\varphi(x),h   \rangle)+\int_0^t P^n_{t-s} \langle Du_n(s,x),Ah+b'_n(x)h  \rangle ds,
\end{equation}
which coincides with \eqref{e40g}. Letting $n\to\infty$, yields \eqref{e40}.\medskip

{\it Step 2}. Conclusion.

Integrating  \eqref{e40} with respect to $\nu$ over $H$ and taking into account   the invariance of $\nu$, yields
\begin{equation}
\label{e41}
\begin{array}{l}
  \ds \int_H\langle D\varphi(x),h\rangle)\nu(dx)= \int_H\langle DP_t\varphi(x), h\rangle \,\nu(dx)
  \\
  \\
 \ds -   \int_H\int_0^t\langle b'(x)h  ,D P_s\varphi(x)\rangle\,ds\,\nu(dx)=:J_1+J_2
  \end{array}
\end{equation}
Setting and $t=1$ we deduce
\begin{equation}
\label{e42}
|J_1|\le \left| \int_H\langle D_xP_t\varphi(x),h\rangle\,\nu(dx)   \right|\le 2C^1_p\,\|\varphi\|_{L^p(H,\nu)}\;|h|.
\end{equation} 
Concerning $J_2$ we have  by \eqref{e37g} and taking into account \eqref{e1.3}
\begin{equation}
\label{e43}
\begin{array}{l}
\ds|J_2|\le \int_0^t\int_HC^1_p(1+(t-s)^{-1/2})\,\|\varphi\|_{L^p(H,\nu)}\;\|b'(\cdot)h\|_{L^q(H,\nu)}\;ds\\
\\
\ds\le K\int_0^t\int_HC^1_p(1+(t-s)^{-1/2})\,\|\varphi\|_{L^p(H,\nu)}\;\|(1+|x|^{2N})\|_{L^q(H,\nu)}\;ds\,|h|.
\end{array}
\end{equation}
 Finally, recalling \eqref{e2.3} and setting $t=1$ the conclusion follows.
\end{proof}

\subsection{Consequences of  the  integral  inequality \eqref{e39}}

The following result can be proved exactly as in \cite{DaDe14}, replacing $R$ by $I$ so,  we  omit  the proof.
\begin{Proposition}
\label{p4.1}
Assume Hypothesis \ref{h1} and let $\nu$ be the invariant measure of problem \eqref{e1.1}. 
Then for any $p>1$ the gradient
$$
D:C^1_b(H)\subset L^p(H,\nu)\to L^p(H,\nu;H),\quad \varphi\to D\varphi,
$$
is closable. 
\end{Proposition}
For any $p>1$ we shall  denote  by $D_p$  the closure of $D$ and by $D_p^*$ the adjoint operator of $D_p$. $D_p$ is a mapping
$$
D_p: D(D_p)\subset L^p(H,\nu)\to L^p(H,\nu;H)
$$
and $D_p^*$ is a mapping
$$
D_p^*:D(D_p^*)\subset L^q(H,\nu;H)\to L^q(H,\nu),
$$
where $q=\frac1{1-p}$.
We have obviously
\begin{equation}
\label{e4.1a}
 \int_H\langle D_p\varphi,F   \rangle\,d\nu =\int_H \varphi\,D_p^*(F)\,d\nu,
\end{equation}
for any $\varphi\in D(D_p)$ and any $F\in D(D_p^*)$. We recall that $F\in D(D_p^*)$ if and only if there exists a positive constant $K_F$ such that
\begin{equation}
\label{e4.2a}
\left| \int_H\langle D\varphi,F   \rangle\,d\nu    \right|\le K_F\|\varphi\|_{L^p(H,\nu)},\quad\forall\;\varphi\in C^1_b(H).
\end{equation}
In this case we have
\begin{equation}
\label{e4.3a}
\|D_p^*(F)\|_{L^q(H,\nu)}\le K_F.
\end{equation}
If no confusion may arise we shall omit sub--indices $p$ in $D_p$ and $D^*_p$.

\begin{Proposition}
\label{r4.2h}
For any $z\in H$ there is $v_z\in L^q(H,\nu)$ for all $ q\in[1,+\infty)$ such that
\begin{equation}
\label{e4.1}
 \int_H\langle D\varphi,z   \rangle\,d\nu =\int_Hv_z\,\varphi\,d\nu.
\end{equation}
\end{Proposition}
\begin{proof}
Let  $z\in H$  and set $F_z(x)=z,\quad\forall\;x\in H.$
Then by \eqref{e39} it follows that
\begin{equation}
\label{e4.4a}
\left|\int_H\langle D\varphi,F_z   \rangle\,d\nu\right|\le C_{1,p}\,\|\varphi\|_{L^p(H,\nu)}\,|z|
\end{equation}
This implies $F_z\in  D(D^*_p)$  and $\|D^*_p(F_z)\|_{L^q(H,\nu)}\le C_{1,p}|z|.$ Setting
$D^*_q(F_z)=v_z,$
identity  \eqref{e4.1} follows.
\end{proof}
  
 \begin{Remark}
 \em  By Proposition \ref{r4.2h} $\nu$ possesses the Fomin derivative  of $\nu$ at the  direction $z$ which  is given precisely by $v_z$ and so, it  belongs to  $L^q(H,\nu)$ for all  $q\in [1,\infty)$,
  \end{Remark}
 Now we are going to identify $v_z$. 
  \begin{Proposition}
\label{p4.2}
Assume Hypothesis \ref{h1}. Then for any $z\in H$ we have   $v_z=\langle D\log\rho,z   \rangle$, where $\rho$ is the   density  of $\nu$  with respect to the Lebesgue measure on $\R^d$. Therefore $\langle D\log\rho,z   \rangle$ belong to $ L^p(H,\nu)$ for any $p\in[1,+\infty)$.
\end{Proposition}
\begin{proof}
First notice that by \eqref{e4.1} it follows in particular  that
\begin{equation}
\label{e4.6h}
\left| \int_H\langle D\varphi,z   \rangle\,d\nu\right| \le \|v_z\|_{L^1(H,\nu)}\,\|\varphi\|_\infty.
\end{equation}
Therefore, by an argument due to Malliavin, $\nu$ has a density $\rho$ with respect to the  Lebesgue measure on $\R^d$ with $\rho\in L^{\frac{d}{d-1}}(\R^d)$, see  \cite{Nu95}.

To prove the last statement, we write \eqref{e4.1} as
$$
 \int_H\langle D\varphi,z   \rangle\,\rho\, dx =\int_H\varphi\,v_z\, \rho\,dx.
$$
This implies in the sense of distributions that
$$
v_z=\langle D\log\rho,z   \rangle.
$$
Now the conclusion follows from Proposition \ref{r4.2h}.

 \end{proof}
  
\begin{Remark}
\em The fact that $\nu$ has a density $\rho$ with respect to the Lebesgue measure, together with several     properties of $\rho$   have  already been  proved  in \cite{MePaRh05},  \cite{BoKrRo01} and \cite{BoKrRo05}.\medskip

\end{Remark}

Let us finally study some properties of operators $D^*$ and $D^*D$.
 \begin{Proposition}
\label{p3.10l}
Let
\begin{equation}
\label{e3.40l}
F(x)=\sum_{h=1}^d f_h(x)e_h,\quad x\in H,
\end{equation}
where $(e_1,...,e_d)$ is an orthonormal basis in $H$ and $f_h\in C^1_b(H)$, $h=1,...,d$. Then $F$  belongs to the domain  of $D^*$ and it results
\begin{equation}
\label{e3.41l}
D^*(F)=-\mbox{\rm div}\;F+\sum_{h=1}^d v_{e_h}\,f_h.
\end{equation}
Moreover, if  $\varphi\in C^2_b(H)$  we have
\begin{equation}
\label{e3.42l}
-\frac12\,D^*D(\varphi)=\frac12\,\Delta\varphi-\frac12\,\sum_{h=1}^d v_{e_h}\,D_{h}\varphi.
\end{equation}
\end{Proposition}
\begin{proof}
Write
\begin{equation}
\label{e3.43l}
\begin{array}{l}
\ds \int_H\langle  D\varphi,F  \rangle\,d\nu=
\sum_{h=1}^d \int_H D_{h}\varphi\,f_h   \,d\nu\\
\\
\ds=\sum_{h=1}^d \int_H D_{h}(\varphi\,f_h)   \,d\nu-\sum_{h=1}^d \int_H \varphi\,D_{h}f_h   \,d\nu
\end{array} 
\end{equation}
Since, in view of \eqref{e4.1}
$$
\int_H D_{h}(\varphi\,f_h)   \,d\nu=\int_H\varphi\,v_{e_h}f_h   \,d\nu,
$$
\eqref{e3.41l} follows. Now \eqref{e3.42l} follows as well setting $F=D\varphi$ in \eqref{e3.41l}.

\end{proof}

  \subsection*{Acknowledgement}  G. Da Prato is partially supported by GNAMPA from INDAM.

\end {document}